\documentclass[twoside,a4paper,12pt]{article}
\usepackage{amssymb, amsmath, amsthm}
\usepackage{fancyhdr}

\newdimen\myMargin
\textwidth \paperwidth
\advance\textwidth -2\myMargin
\textheight \paperheight
\advance\textheight -2\myMargin
\advance\oddsidemargin \myMargin
\advance\evensidemargin \myMargin
\advance\topmargin \myMargin
\advance\textheight -17 true mm
\advance\topmargin 12.5 true mm

\newcommand{\F}{\mathbb F}
\newcommand{\stb}[2]{#1_1,\dots,#1_{#2}}

\newcommand{\N}{\mathbb N}

\newcommand{\ve}[1]{\mathbf{#1}}
\newcommand{\Fx}{\F\left[x_1,\dots,x_n\right]}
\newcommand{\Fxv}{\F\left[\ve x\right]}

\newcommand{\monom}[2]{\ve{#1}^{\ve{#2}}}

\newtheorem{theorem}{Theorem}
\newtheorem{lemma}[theorem]{Lemma}
\newtheorem{corollary}[theorem]{Corollary}

\newtheorem{definition}[theorem]{Definition}

\theoremstyle{remark}
\theoremstyle{remark}\newtheorem{remark}[theorem]{Remark}

\title{Alon's Nullstellensatz for multisets}

\author{\normalsize
  \begin{minipage}{0.4\linewidth}
    \large
    G\'eza K\'os \\
    \footnotesize
    Computer and Automation Research Instute,
    Hungarian Academy of Sciences; \\
    Department of Analysis,
    E\"otv\"os Lor\'and University, Budapest \\
    \texttt{kosgeza@sztaki.hu} \\
    \normalsize
  \end{minipage}
  \qquad
  \begin{minipage}{0.4\linewidth}
    \large
    Lajos R\'onyai \\
    \footnotesize
    Computer and Automation Research Instute,
    Hungarian Academy of Sciences; \\
    Department of Algebra,
    Budapest University of Technology and Economics \\
    \texttt{ronyai@sztaki.hu}
    \normalsize
  \end{minipage}
  \vspace{0.5cm}
}

\begin{document}

\thispagestyle{empty}

\maketitle

\footnotetext{
\noindent
{\em Mathematics Subject Classification (MSC2010):} 05-XX, 05E40, 12D10. 
Research supported in part by OTKA grants NK 72845, K77476, and K77778. 
}


\begin{abstract}
Alon's combinatorial Nullstellensatz (Theorem 1.1 from \cite{Alon1}) is one
of the most powerful algebraic tools in combinatorics, with a diverse array
of applications.
Let $\F$ be a field, $S_1,S_2, \ldots , S_n$ be finite nonempty subsets
of $\F$. 
Alon's theorem is a specialized, precise
version of the Hilbertsche Nullstellensatz for the ideal of all polynomial
functions
vanishing on the set $S=S_1\times S_2\times \cdots \times S_n\subseteq
\F^n$. From
this  Alon deduces a simple and amazingly widely applicable nonvanishing
criterion (Theorem 1.2 in \cite{Alon1}). It provides a sufficient condition
for a
polynomial $f(x_1,\ldots ,x_n)$ which guarantees that $f$ is not identically zero 
on the set $S$. In this paper  we extend these two results from sets of points to 
multisets. We give two different proofs of the generalized nonvanishing
theorem. We extend some of the known applications of the original 
nonvanishing theorem to a setting allowing multiplicities, including the 
theorem of Alon and F\"uredi on the hyperplane coverings of discrete cubes.  

\end{abstract}


\section{Introduction}

Alon's combinatorial Nullstellensatz (Theorem 1.1 from \cite{Alon1}) is one
of the most powerful algebraic tools in combinatorics. It has dozens of
beautiful and strong applications, see \cite{CLM}, \cite{Felszeghy},
\cite{GT}, \cite{KarolyiR}, \cite{PS}, \cite{Sun} for some recent examples. 

Let $\F$ be a field, $S_1,S_2, \ldots , S_n$ be finite nonempty subsets 
of $\F$. Let $\Fxv=\Fx$ stand for the ring of polynomials over $\F$ in
variables $x_1, \ldots, x_n$.  Alon's theorem is a specialized, precise 
version of the Hilbertsche Nullstellensatz for the ideal of all polynomial 
functions
vanishing on the set $S=S_1\times S_2\times \cdots \times S_n\subseteq
\F^n$, and for the basis  $f_1,f_2,\ldots ,f_n$, where 
$$ f_i=f_i(x_i)=\prod_{s\in S_i}(x_i-s)\in \Fxv $$ 
for $i=1,\ldots ,n$. From 
this  Alon deduces a simple and amazingly widely applicable nonvanishing 
criterion (Theorem 1.2 in \cite{Alon1}). It provides a sufficient condition for a
polynomial $f\in \Fxv$ which guarantees that $f$ is not identically zero on 
$S$. Here we aim to extend these two results from sets of points to 
multisets.

\bigskip

To formulate our results, we need some more notation and definitions.
Let $\N$ denote the set of nonnegative integers, and let $n$ be a
fixed positive integer. Vectors of length $n$ are denoted by boldface
letters, for example $\ve s=(\stb sn)\in\F^n$ stands for points in the
space $\F^n$. For vectors $\ve a,\ve b\in \N^n$, the relation $\ve
a\ge\ve b$ etc. means that the relation holds at every component. We
use the same notations for constant
vectors. e.g. $\ve0=(0,0,\ldots,0)$ or $\ve1=(1,1,\ldots,1)$.

For $\ve w\in\N^n$, we write $\monom xw$ for the monomial
$x_1^{w_1}\dots x_n^{w_n}\in\Fxv$.  If $\ve s\in\F^n$, then $(\ve
x-\ve s)^{\ve w}$ stands for the polynomial
$(x_1-s_1)^{w_1}\dots(x_n-s_n)^{w_n}$.

\bigskip

It is well known that for an arbitrary $\ve s\in \F^n$ we can express
a polynomial $f(\ve x)\in \Fxv$ as
\begin{equation} \label{kifejt} f(\ve x) = \sum_{\ve u\in \N^n} f_{\ve u}(\ve
  s) (\ve x - \ve s)^{\ve u},
\end{equation} 
where the coefficients $f_{\ve u}(\ve s)\in \F$ are uniquely
determined by $f$, $\ve u$ and~$\ve s$.  In particular we have
$f_{\ve0}(\ve s)=f(\ve s)$ for all $\ve s \in\F^n$. If
$u_i<\mathrm{char}\,\F$ for all $i$, then we have
$$
f_{\ve u}(\ve s) = \frac1{u_1!\cdots u_n!} \cdot
\frac{\partial^{u_1\dots+u_n} }{\partial x_1^{u_1}
  \dots\partial x_n^{u_n}} f(\ve s).
$$
Notice also that if $u_1+\dots+u_n\ge\deg f$, then $f_{\ve u}=f_{\ve
  u}(\ve s)$ does not depend on $\ve s$.

\bigskip

For a point $ \ve s \in \F^n$ and an exponent vector $\ve w \in \N^n$
with positive integer components we write $I(\ve s , \ve w )$ for the
set of polynomials $f(\stb xn )$ for which in the expansion
\eqref{kifejt} we have $f_{\ve u}(\ve s)=0$ for all $\ve u<\ve w$.
It is a simple matter to check that $I(\ve s , \ve w )$ is actually an
ideal in $\Fxv$. We have also that

\begin{equation}
\label{dim-eq}
\dim_{\F}\Fxv / I(\ve s , \ve w )=w_1w_2\cdots w_n,
\end{equation}
because the monomials $(\ve x-\ve s)^{\ve u}$ with $0\leq u_j<w_j$ form
a basis of the factor $\Fxv / I(\ve s , \ve w )$.

\medskip

As before, suppose that  $S_1, S_2, \ldots ,S_n$ are nonempty finite subsets 
of $\F$. Suppose further that we have 
a positive integer {\em multiplicity} $m_i(s)$ attached to the elements of 
$s\in S_i$. This way we can view the pair $(S_i,m_i)$ as a multiset which
contains the element  $s\in  S_i$ precisely $m_i(s)$ times. We shall
consider the sum $d_i=d(S_i):=\sum\limits_{s\in S_i}m_i(s)$ as the size 
of the multiset $(S_i,m_i)$.
We put 
$S=S_1\times S_2\times \cdots S_n$.
For an element $\ve s =(\stb sn) \in S$ we set the multiplicity vector 
$m(\ve s)$ as $(m_1(s_1),\ldots ,m_n(s_n) )$, and write 
$|m(\ve s)|=m_1(s_1)+\cdots +m_n(s_n)$. 

\medskip

Our principal object of interest is the ideal 
$$ I=I(S)=\bigcap_{\ve s\in S} I(\ve s, m(\ve s)). $$
For $i=1,\ldots, n$ we define the polynomials $g_i(x_i)\in \Fxv $ as 
\begin{equation} \label{g-eq}
g_i(x_i)=\prod_{s\in S_i}(x_i-s)^{m_i(s)}.
\end{equation}
We see that $g_i$ is a monic polynomial of degree  $d_i$. Moreover, 
for the ideal generated by the $g_i$ we have 
\begin{equation}\label{contain}
(g_1(x_1),g_2(x_2), \ldots ,g_n(x_n))\subseteq I.
\end{equation}

\bigskip

The following theorem is a generalization of Alon's Nullstellensatz (Theorem
1.1 from \cite{Alon1}). We recover Alon's result by setting $m_i(s)=1$
everywhere.

\begin{theorem} \label{nullstellensatz}
We have 
$$ (g_1(x_1),\ldots ,g_n(x_n)) = I.$$ 
Moreover, for every polynomial $f(\ve x)\in \Fxv$ there are polynomials 
$h_1,\ldots ,h_n, r\in \Fxv $ such that $\deg h_i\leq \deg f - d_i$, the 
degree of $r$ is less than $d_i$ in every $x_i$, for which 
$$ f(\ve x)=r(\ve x)+ \sum_{i=1}^n h_i(\ve x)g_i(x_i ). $$ 
In the above expansion $r$ is uniquely determined by $f$. 
\end{theorem}

\begin{remark}
  We have $r(\ve x) \equiv 0$ in the expansion of the theorem if and
  only if $f\in I$.
\end{remark}

We can strengthen a little the part of Theorem~\ref{nullstellensatz}
which states that
$\{g_1, \ldots , g_n \}$ is a nice generating set for $I$. For the basics
of the theory of Gr\"obner bases we refer to \cite{tapas} and \cite{AL}. 

\begin{corollary} 
\label{groebner}
The set of polynomials 
$\{g_1, \ldots , g_n \}$ is a universal Gr\"obner basis for $I$.
\end{corollary}

\smallskip

\begin{remark} This will follow easily from the proof of
  Theorem~\ref{nullstellensatz}. The Gr\"obner basis property of
  $\{g_1, \ldots , g_n \}$ for the ideal it generates can also be
  proved by applying directly and very simply the $S$-polynomial test
  of Buchberger, (cf. \cite{B}, and Theorem 3.10 of Chapter 1 from
  \cite{tapas}) to the pair of polynomials $g_i(x_i), g_j(x_j)$.
\end{remark}

\begin{remark} 
\label{subring}
As in the case of Alon's theorem, we have that 
if the coefficients of 
$f$ and $g_i$ are from some subring $R$ of $\F$, then the polynomials 
$h_i$ and $r$ will be from $R[\stb xn]$ as well. 
\end{remark}

We can now formulate a version of Alon's powerful nonvanishing theorem 
(Theorem
1.2 in \cite{Alon1}) for multiple points.  Again, we obtain  
Alon's result by setting $m_i(s)=1$ identically.

\begin{theorem} \label{nonvanish}
Let $\F$ be a field, $f=f(x_1,\ldots ,x_n)\in \F [x_1,\ldots ,x_n]$ be a
polynomial of degree $\sum\limits_{i=1}^n t_i$, where each $t_i$ is a nonnegative
integer. Assume, that the coefficient in $f$ of the monomial 
$x_1^{t_1}x_2^{t_2}\cdots x_n^{t_n}$ is nonzero.  Suppose further that
$(S_1, m_1), (S_2,m_2), \ldots ,(S_n, m_n)$ are multisets of $\F$ such that 
for the size $d_i$ of $(S_i,m_i)$ we have $d_i>t_i$ ($i=1, \ldots , n$).
Then $f$ is not in the ideal $I$ attached to the multisets $(S_i,m_i)$. 

In
other words, there exists a point $\ve s =(s_1, \ldots, s_n)\in S_1\times \cdots
\times S_n$ and an exponent vector $\ve u =(u_1, \ldots, u_n)$ with 
$u_i<m_i(s_i)$ for each $i$, such that $f_{\ve u}(\ve s) \not = 0$ 
in the expansion of $f$ as
$$f(\stb xn )=\sum f_{\ve u}(\ve s)  (\ve x-\ve s)^{\ve u}, ~~
f_{\ve u}(\ve s) \in \F.$$
\end{theorem}

\medskip

For two multisets $(H_1,m_1),(H_2,m_2)$ we write    
$(H_1,m_1)\subseteq (H_2,m_2)$ if $m_1(h)\leq m_2(h)$ holds whenever 
$h\in H_1$. We call the multisubset $(H_1,m_1)\subseteq (H_2,m_2)$ a tight 
multisubset, if  $m_1(h)=m_2(h)$ holds for every $h\in H_1$. 

In \cite{BS} Ball and Serra proved a punctured version of Alon's
Nullstellensatz. The result and the proof extends with slight
modifications to the a multiset case.

Let $(S_1, m_1),\ldots , (S_n, m_n)$ be multisets from the field $\F$.
Suppose that $(D_i,m_i)$ is a nonempty tight multisubset of $(S_i,m_i)$ for
$i=1,\ldots ,n$. Write $D=D_1\times D_2\times \cdots \times D_n$.
Let $g_i(x_i)$ be the polynomials from \eqref{g-eq} and put 
\begin{equation} \label{l-eq}
\ell_i(x_i)=\prod_{s\in D_i}(x_i-s)^{m_i(s)} \mbox{ for } i=1,\ldots ,n.
\end{equation}

\begin{theorem} 
\label{ballserra}
Let $f(\ve x)\in \Fxv$ be a polynomial such that $f\in I(\ve
s, m(\ve s))$ for all $\ve s\in S$ with the exception of at least 
one $\ve s^* \in D$, for which 
 $f\not\in I(\ve  s^*, m(\ve s^*))$. Then there are polynomials 
$h_1,\ldots ,h_n, r\in \Fxv $ such that $\deg h_i\leq \deg f - d_i$, the
degree of $r$ is less than $d_i$ in every $x_i$, for which   
$$ f(\ve x)=r(\ve x)+ \sum_{i=1}^n h_i(\ve x)g_i(\ve x), $$
and 
$$ r= h \prod_{i=1}^n\frac{g_i(x_i)}{\ell_i(x_i)} $$
for some nonzero $h\in \Fxv$. As a consequence, $\deg (f)\geq
\sum\limits\limits_{i=1}^n\big(d(S_i)-d(D_i)\big)$.
\end{theorem}
  
\medskip

We mention here one more related result from \cite{BS} by Ball and Serra. They
obtained a generalization of Alon's Nullstellensatz to polynomials which
vanish at least $t$ times at every point of $S$ (cf. Theorem 3.1 in \cite{BS}).
This result is in turn related to the method of multiplicities (see the
paper \cite{DKSS} by Dvir, Kopparty, Saraf and Sudan). To give a specific
example,
from Theorem 3.1 of \cite{BS} it follows immediately that if $S$ is a subset 
of a
field $\F$, $f\in \F[x_1,\ldots ,x_n]$ is a polynomial of degree $d$ which
vanishes at least $t$ times at every point of $S^n$, then
$\text{deg} f\geq
t|S|$. This Schwartz-Zippel type inequality is an important special case of
Lemma 8 from \cite{DKSS}.

In the next section we prove Theorems \ref{nullstellensatz},
\ref{nonvanish}, and~\ref{ballserra}. The proof of
Theorem~\ref{nullstellensatz} uses some very simple facts from
commutative algebra.  For Theorem~\ref{nonvanish} we offer two
different proofs. The first one is a direct application of
Theorem~\ref{nullstellensatz}, while the second proof involves a
little more explicit relation among the expansion coefficients of $f$,
and is based on elementary calculations with divided differences
(Theorem~\ref{alt3:lem:linkomb}). We believe that
Theorem~\ref{alt3:lem:linkomb} is also of independent interest.

Section 3 is devoted to applications. We extend some known applications of
the nonvanishing theorem to a setting allowing multiplicities. In most cases 
the original proofs are generalized to higher multiplicities.


\section{Proofs of Theorems~\ref{nullstellensatz}, \ref{nonvanish},
  and~\ref{ballserra} }

First we prove Theorem~\ref{nullstellensatz}. We use Alon's original 
argument together with dimension counting.

\medskip

\begin{proof}[Proof of Theorem~\ref{nullstellensatz}]
We recall first that  
$$ I=I(S)=\bigcap_{\ve s\in S} I(\ve s, m(\ve s)). $$
We show next that
\begin{equation} \label{Idimenzio}
\dim_{\F} \Fxv\big/ I=d_1d_2\cdots d_n.
\end{equation}

Indeed, the ideals $I(\ve s, m(\ve s))$ are pairwise relatively prime,
as the radicals of $I(\ve s, m(\ve s))$ are the maximal ideals
$(x_1-s_1, \ldots x_n-s_n)$, which are clearly relatively prime (see
Proposition 1.16 in \cite{AM}).  Now the Chinese Remainder Theorem
(Proposition 1.10 in \cite{AM}) gives that
\[
\Fxv\big/I\cong\bigoplus_{\ve s\in S} \Fxv\big/ I(\ve s, m(\ve s)).
\]
By taking dimensions and using \eqref{dim-eq} we obtain
$$ \dim _\F \Fxv\big/ I=\sum_{\ve s\in S} \dim_{\F} \Fxv\big/ I(\ve s,
m(\ve s))=
\sum_{\ve s\in S}m_1(s_1)\cdots m_n(s_n)=d_1d_2\cdots d_n. $$

To establish the Theorem, we focus first on the second statement. In
the monomials occurring in $f$ we repeatedly substitute $x_i^{d_i}-
g_i(x_i )$ for $x_i^{d_i}$ as long as possible. As $ \deg (x_i^{d_i}-
g_i(x_i))< d_i$, this reduction process is guaranteed to terminate in
finite steps with an $r$ of the desired form. Notice also, that the
above reduction step means subtracting a multiple of degree at most
$\deg f - \deg g_i$ of $g_i$ from $f$.  From the degree constraints
for $r$ we obtain the inequality
$$
\dim_{\F} \Fxv\big/ (g_1(x_1),\ldots ,g_n(x_n)) \leq d_1d_2\cdots d_n.
$$
Comparing this with \eqref{Idimenzio} and \eqref{contain}, we see that there
must be an equality in \eqref{contain}, proving the first claim.  

The uniqueness of $r$ also follows since two 
such polynomials $r$ and $r'$ satisfy $r-r'\in I$, and then the degree 
constraints imply that $r-r'=0$. 
\end{proof}

\begin{remark}
  Alternatively, one can prove $\dim_{\F} \Fxv\big/ (g_1(x_1),\ldots
  ,g_n(x_n)) = d_1d_2\cdots d_n$ by a repeated application of the
  following simple fact: if $A$ is a commutative ring and $f(x)\in
  A[x]$ is a monic polynomial of positive degree, then $A[x]/(f)$ is a
  free $A$-module of rank $\deg f$.
\end{remark}

\medskip

\begin{proof}[Proof of Corollary~\ref{groebner}]
Let $\prec$ be an arbitrary term order on the monomials    
of $\Fxv$. We observe that in the course of the reduction of a monomial
$\ve y$,
when we substitute $x_i^{d_i} -g_i(x_i )$ for $x_i^{d_i}$, we replace
$\ve y$ by a linear combination of monomials which are all 
$\prec$-smaller than $\ve y$. This implies in particular, that if $f\in I$
and $0\not= \ve y$ is the $\prec$-largest monomial of $f$, then there
exists an $i$ such that  $x_i^{d_i}\preceq \ve y$.
\end{proof}

\smallskip

From the proof Theorem~\ref{nullstellensatz}
it is apparent that if $f, g_i \in R[x_1, \ldots ,x_n]$ for some
subring $R$ of $\F$, then $r, h_i \in R[x_1, \ldots ,x_n]$ as well, proving
the claim of Remark~\ref{subring}. 
 
\smallskip

Theorem~\ref{nonvanish} now readily follows. The original argument of Alon 
is verbatim applicable, and is reproduced here for the reader's convenience.

\begin {proof}[Proof of Theorem~\ref{nonvanish}]
 Suppose for contradiction that $f\in I=I(S)$.
Then by Theorem~\ref{nullstellensatz} there are polynomials 
$h_1,\ldots ,h_n,\in \Fxv $ such that $\deg h_i\leq \deg f - d_i$,
for which   
$$ f(\ve x)=\sum_{i=1}^n h_i(\ve x)g_i(x_i ),  $$
where $g_i$ are the polynomials from \eqref{g-eq}. 
The coefficient of $x_1^{t_1}x_2^{t_2}\cdots x_n^{t_n}$ on the left is   
nonzero. On the other hand, the degree of $h_ig_i$ is at most the degree 
of $f$, and any monomial of this degree must be divisible by $x_i^{d_i}$ for
some $i$. It follows that the coefficient of 
$x_1^{t_1}x_2^{t_2}\cdots x_n^{t_n}$  is 0 on the right hand side. This is a
contradiction completing the proof.
\end{proof}

\smallskip

Next we adapt the argument of Ball and Serra from \cite{BS} 
to prove Theorem~\ref{ballserra}.

\begin{proof}[Proof of Theorem~\ref{ballserra}]
By Theorem~\ref{nullstellensatz} we can write $f$ as
$$ f(\ve x)=r(\ve x)+ \sum_{i=1}^n h_i(\ve x)g_i(x_i), $$
with $h_1,\ldots ,h_n, r\in \Fxv $,  $\deg h_i\leq \deg f - d_i$, and the
degree of $r$ is less than $d_i$ in every $x_i$. For each $i$ the polynomial
$r\ell_i$ is in $I$, hence it can be reduced to 0 by using the polynomials
$g_1(x_1),\ldots ,g_n(x_n)$. But if $j\not =i$ then $g_j(x_j)$ can not be
used in the reduction of $r\ell_i$ (or of any reduct of $r\ell_i$ by
$g_i(x_i)$) because the degree of  $r\ell_i$ in $x_j$ is less than $d_j$.  
We infer, that $g_i$ divides  $r\ell_i$: there is a polynomial $r_i\in \Fxv$
such that  $r(\ve x) \ell_i(x_i)=g_i(x_i)r_i(\ve x)$. Using that $\ell_i $
divides  $g_i$, we have that $\frac{g_i(x_i)}{\ell_i(x_i)}$ divides $r$.
Knowing that $\Fxv$ is a UFD and $\frac{g_i(x_i)}{\ell_i(x_i)}$ and
$\frac{g_j(x_j)}{\ell_j(x_j)}$ have
no associate prime factors in $\Fxv$ for $i\not= j$, we obtain that
$$ r= h \prod_{i=1}^n\frac{g_i(x_i)}{\ell_i(x_i)}, $$
with some polynomial $h$. Here  $h\not=0$ because $f\not\in I$ and hence
$r\not= 0$. The last statement follows from $\deg f\geq \deg r$.

\end{proof}


\subsection{An alternative proof for Theorem~\ref{nonvanish}}

Our objective here is to give a more direct proof of
Theorem~\ref{nonvanish}. It is based on a linear relation 
among the expansion coefficients of $f$, which we develop in 
Theorem~\ref{alt3:lem:linkomb}.

Throughout this subsection we keep our standard notation: $(S_1, m_1),
(S_2,m_2), \ldots ,(S_n, m_n)$ are nonempty finite multisets from
$\F$, and $d_i$ denotes the size of the multiset $(S_i,m_i)$. We put
$S=S_1\times\cdots\times S_n\subset\F^n$. We set also
$g_i(x_i)=\prod\limits_{s\in S_i}(x_i-s)^{m_i(s)}$ for $i=1,\ldots ,
n$, and $g(\ve x)=\prod\limits_{i=1}^ng_i(x_i)$.

\begin{theorem}
  \label{alt3:lem:linkomb} Let $\ve t=\ve d(S)-\ve1=
  \big(d_1-1,\ldots, d_n-1\big)$.

  (a) Then there exist constants
  $\alpha^{(\ve s)}_{\ve u}\in \F$ for $\ve s\in S$, $\ve u<m(\ve s)$,
independent of $f$,  such that
  \begin{equation}
    \label{alt3:eq:linkomb}
    f_{\ve t} = \sum_{\ve s \in S} \sum_{\ve u < m(\ve s)} \alpha^{(\ve
      s)}_{\ve u} f_{\ve u} (\ve s)
  \end{equation}
  holds for all polynomials $f\in\Fxv$ with $\deg f\le
  t_1+\dots+t_n$.

  (b) The coefficients $\alpha^{(\ve s)}_{\ve u}$ are uniquely
  determined by $(S,m)$, $\ve s$, $\ve t$ and $\ve u$.

  (c) If $\ve s\in S$ and $\ve u=m(\ve s)-\ve1$, then $\alpha^{(\ve
    s)}_{\ve u}\neq0$.
\end{theorem}

\bigskip 

To prove Theorem~\ref{alt3:lem:linkomb}, we apply some well-known
properties of divided differences of univariate polynomials (see \cite{dB}). 
Our considerations include finite fields as well, where these facts must be
handled with special care. In the statement above we allow
multiplicities beyond the field characteristics, and many difficulties
arise when one works with derivatives of order higher than the characteristics. 
Thus, for the sake of completeness, we re-build some of the classical facts 
on divided differences, but without any recourse to derivatives.

We will use also the uniqueness of the polynomial $r$ in the second
statement of Theorem~\ref{nullstellensatz}. We will use the notation
$$ h = (f\mod (g_1,\dots,g_n)) $$
for the unique $h\in\Fxv$ such that $f-h\in(g_1,\dots,g_n)$ and
$\deg_i h<d_i$ for every~$i$.


\begin{definition}
  \label{alt3:def:dd}
  For  $f\in\Fxv$ we denote by $f[S]$ the coefficient of ${\ve
    x}^{ \ve d(S)-\ve1}=x_1^{d_1-1}\cdots x_n^{d_n-1}$ in the
  polynomial $(f\mod (g_1,\dots,g_n))$.
\end{definition}


\begin{lemma}
  \label{alt3:lem:ddprop} ~
  Let $f\in\Fxv$ be a polynomial over $\F$.
  
\noindent
  (a) If every $S_i$ consists of a single element $a_i$ with
  multiplicity $t_i+1$, then $f[S]=f_{\ve t}((\stb an))$.

  \noindent
  (b) Suppose that some $S_i$ contains at least two different
  elements, say $a$ and $b$. Let $S_i'=S_i\setminus\{a\}$ and
  $S_i''=S_i\setminus\{b\}$ (these multisets contain $a$ and $b$ with
  multiplicity one less than $S_i$), and $S'=S_1\times\cdots\times
  S_{i-1}\times S_i'\times S_{i+1}\times\cdots\times S_n$ and
  $S''=S_1\times\cdots\times S_{i-1}\times S_i''\times
  S_{i+1}\times\cdots\times S_n$. Then
  $$ f[S] = \frac{f[S']-f[S'']}{b-a}. $$
\end{lemma}

\bigskip 

\begin{proof}
  To prove part (a), observe that
  $$
  \Big( f(\ve x) \mod ((x_1-a_1)^{t_1+1},\dots,(x_n-a_n)^{t_n+1}) \Big) =
  \sum_{\ve u\le\ve t}f_u(\ve a) (\ve x-\ve a)^{\ve u}.
  $$
  Then the coefficient of $\ve x^{\ve t}$ is $f[S]$ on the left-hand
  side, and it is $f_{\ve t}(\ve a)$ on the right-hand side.

  As for part (b), from the definition we see that 
  \begin{gather*}
    (x_i-a)\left(f(\ve x) \mod (g_1(x_1),\dots,g_{i-1}(x_{i-1}),
    \frac{g_i(x_i)}{x_i-a}, g_{i+1}(x_{i+1}),\dots,g_n(x_n) )\right) - \\ -
    (x_i-b)\left(f(\ve x) \mod (g_1(x_1),\dots,g_{i-1}(x_{i-1}),
    \frac{g_i(x_i)}{x_i-b}, g_{i+1}(x_{i+1}),\dots,g_n(x_n)) \right) = \\
    \Big( (x_i-a)f(\ve x) \mod (g_1(x_1),\dots,g_n(x_n)) \Big) -
    \Big( (x_i-b)f(\ve x) \mod (g_1(x_1),\dots,g_n(x_n)) \Big) = \\
    \Big( (b-a)f(\ve x) \mod (g_1(x_1),\dots,g_n(x_n)) \Big).
  \end{gather*}
  Comparing the coefficients of ${\ve x}^{\ve t}$, we obtain
  $$
  f\big[S'\big] - f\big[S''\big] = (b-a) f[S].
  $$
\end{proof}

\bigskip 

\begin{proof}[Proof of Theorem~\ref{alt3:lem:linkomb}]
  (a) By Definition~\ref{alt3:def:dd}, we have
  $$ f_{\ve t}=f[S]. $$
  Apply Lemma~\ref{alt3:lem:ddprop}(b) to the right-hand side
  repeatedly as long as possible. At the end, we arrive at a linear
  combination of some terms of the form $f[M]$ where
  $M=M_1\times\dots\times M_n\subset S$ such that each $M_i$ consist of a
  single element $s_i$ with some multiplicity $u_i+1\leq m_i(s_i)$. By
  Lemma~\ref{alt3:lem:ddprop}(a), we have $f[M]=f_{\ve u}(\ve s)$.

  (b) Suppose that there exist two different systems of constants,
  $(\alpha_{\ve u}^{(s)})$ and $(\alpha'{}_{\ve u}^{(s)})$ which
  have the properties described in part (a). Taking the
  differences, $\delta_{\ve u}^{(s)}=\alpha_{\ve
    u}^{(s)}-\alpha'{}_{\ve u}^{(s)}$ we have
  \begin{equation}
    \label{eq:delta}
    \sum_{\ve s \in S} \sum_{\ve u < m(\ve s)} \delta^{(\ve
      s)}_{\ve u} f_{\ve u} (\ve s) = 0
  \end{equation}
  for all polynomials $f\in\Fxv$, with  $\deg f\le t_1+\dots+t_n$.

  Since the systems $(\alpha_{\ve u}^{(s)})$ and $(\alpha'{}_{\ve
    u}^{(s)})$ are different, there exists some $\delta_{\ve u}^{(\ve
    s)}$ which is not $0$. Take such a $\delta_{\ve u}^{(\ve s)}$
  where the vector $\ve u$ is maximal. Apply
  \eqref{eq:delta} to the polynomial
  $$
  f(\ve x) = \prod_{i=1}^n \left( (x_i-s_i)^{u_i} \prod_{r\in
      S_i\setminus\{s_i\}} (x_i-r)^{m_i( r)} \right).
  $$
  Then, on the left-hand side of \eqref{eq:delta}, since  
$f_{\ve u'}(\ve s)=0$ unless $\ve u'\geq \ve u$, we see that
$\delta_{\ve
    u}^{(\ve s)} f_{\ve u}(\ve s)$ is the only nonzero
  term, giving a  contradiction.

  (c) Fix $\ve s$ and $\ve u=m(\ve s)-\ve1$. Again, apply
  Lemma~\ref{alt3:lem:ddprop}(b) repeatedly to compute 
   $f[S]$. Whenever we have some
  different $s_i$ and $b$ in $S_i$, apply
  Lemma~\ref{alt3:lem:ddprop}(b) to that pair. This way the term
  $f_{\ve u}(\ve s)$ is obtained only once, and with a nonzero
  coefficient. In fact, we obtain that 
$$ \alpha _{\ve u}^{(\ve s)}=\prod _{i=1}^n\prod_{s\in S_i\setminus \{s_i\}}
\frac{1}{(s-s_i)^{m(s)}}. $$
\end{proof}

\bigskip 

\begin{proof}[Alternative proof of Theorem~\ref{nonvanish}]
  If $d(S_i)>t_i+1$ for some $i$, then we can remove an element from 
$S_i$ (or decrease its multiplicity). So we can assume that 
$d(S_i)=t_i+1$ for every $i$. 

  Apply Theorem~\ref{alt3:lem:linkomb}. On the left-hand side of
  \eqref{alt3:eq:linkomb}, the coefficient $f_{\ve t}$ is not zero.
  Hence, at least one of the values $f_{\ve u}(\ve s)$ is 
different from zero.
\end{proof}


\section{Applications}

Some of the known applications of Alon's  nonvanishing theorem 
can be extended to multisets. Typically we found that the original argument
can be modified to allow higher multiplicities. 


\subsection{Covering cubes}

We can extend a result of Alon and F\"uredi \cite{AF} on the  covering of a
discrete cube by hyperplanes in the following way. 

\begin{theorem}
  Let $(S_1, m_1), \ldots ,(S_n,m_n)$ be finite multisets from the
  field $\F$. Suppose that $0\in S_i$, with $m_i(0)=1$ for every $i$,
  and $H_1,\dots,H_k$ are hyperplanes in
  $\F^n$ such that every point $\ve s\in S\setminus\{\ve0\}$ is covered 
by at least $|m(\ve s)|-n+1$ hyperplanes and 
the point $\ve0$ is not covered by any of the hyperplanes.
Then $k\geq d(S_1)+d(S_2)+\cdots + d(S_n)-n$.
\end{theorem}

We give three proofs.  The first of them is essentially the original proof 
of Alon and F\"uredi (see \cite{AF}, \cite{Alon1}), adapted to the multiple 
point setting. The second proof uses Theorem~\ref{alt3:lem:linkomb}
directly. The last one is a quite straightforward application of the 
generalized Ball-Serra theorem.

\begin{proof}[First proof] Let  $\ell_j(\ve x)$ be
the
linear polynomial defining the hyperplane
  $H_j$, set  $f(\ve x)=\prod\limits_{j=1}^k\ell_j(\ve x)$, and
$t_i=d(S_i)-1$.

  Let
  $$
  P(\ve x) = \prod_{i=1}^n\prod_{s\in S_i\setminus \{0\}}(x_i-s)^{m_i(s)}
  $$ 
  and
  $$
  F(\ve x) = P(\ve x) - \frac{P(\ve0)}{f(\ve0)}f(\ve x).
  $$
  Note that we have $f(\ve0)\ne0$, because the
  hyperplanes do not cover $\ve0$.
  If the statement is false, then the degree of $F$ is $t_1+t_2+\cdots
  + t_n$ and the coefficient of $x_1^{t_1}\cdots x_n^{t_n}$ is
  $1$. Theorem~\ref{nonvanish} applies for $(S_1, m_1), \ldots
  ,(S_n,m_n)$ and $t_1,\ldots ,t_n$: there exists a vector $\ve s \in  
  S$ such that $F \not \in I(\ve s, m(\ve s))$. We observe that $\ve   
  s$ can not be $\ve 0$, because $F(\ve 0)=0$. Thus $\ve s$ must have at
  least one nonzero coordinate, implying that
  $$
    P(\ve x)  \in  I(\ve s, m(\ve s)).
  $$
Moreover, as $\ve s$ is a nonzero vector,  $f(\ve x)$ must vanish
  at $\ve s$ at least $|m(\ve s)|-n+1$ times, implying that 
  $f(\ve x) \in I(\ve s, m(\ve s))$ (expand the product at $\ve s$; for
  every term $(\ve x -\ve s )^{\ve u} $ obtained there will be an
  index $j$ such that $u_j\geq m_j(s_j)$). From $P(\ve x), f(\ve x) \in  
  I(\ve s, m(\ve s))$ we infer that  $F(\ve x)\in  I(\ve s, m(\ve s))$.
  This contradiction finishes the proof.
\end{proof}

\bigskip 

\begin{remark} The polynomial
  $\prod_{i=1}^n\prod_{s\in S_i\setminus \{0\}}(x_i-s)^{m_i(s)}$
  used in the preceding argument shows that the bound of the theorem
  is sharp for any selection of $(S_i,m_i)$. It gives
  $d(S_1)+d(S_2)+\cdots + d(S_n)-n$ hyperplanes with the required
  covering multiplicities.
\end{remark}

\begin{proof}[Second proof]
   We keep the notation $t_i=d(S_i)-1$. We have 
$$d(S_1)+d(S_2)+\cdots + d(S_n)-n = t_1+\dots+t_n.$$
  As in the first proof, let $\ell_j(\ve x)$ be the linear
  polynomial defining the hyperplane $H_j$, and  $f(\ve
  x)=\prod\limits_{j=1}^k\ell_j(\ve x)$. Our goal is to prove $\deg
  f\ge t_1+\dots+t_n$.

  Suppose that $k=\deg f< t_1+\dots+t_n$. By
  Theorem~\ref{alt3:lem:linkomb} we have
  $$
  f_{\ve t} = \sum_{\ve s \in S} \sum_{\ve u < m(\ve s)}
  \alpha^{(\ve s)}_{\ve u} f_{\ve u} (\ve s).
  $$
  On the right-hand side, we have $f_{\ve u} (\ve s)=0$ for all $\ve
  s\in S\setminus\{\ve0\}$ and $\ve u<m(\ve s)$.

  Since the point $\ve0$ is not covered, we have
  $f(\ve0)=f_{\ve0}(\ve0)\ne0$ and, by
  Theorem~\ref{alt3:lem:linkomb}(c),
  $\alpha^{(\ve0)}_{\ve0}\ne0$. Therefore,
  $$
  f_{\ve t} = \alpha^{(\ve0)}_{\ve0} \cdot f_{\ve0}(\ve0) \neq 0.
  $$
  But $f_{\ve t}\neq0$ is possible only if $\deg f\ge
  t_1+\dots+t_n$.
\end{proof}

\medskip 

\begin{proof}[Third proof]
We can apply Theorem~\ref{ballserra} directly with $D_i=\{0\}$, $m_i(0)=1$,
$\ve s^*=\ve 0$, and $f(\ve x)=\prod\limits_{j=1}^k\ell_j(\ve x)$.
\end{proof}

\medskip


\subsection{The Cauchy-Davenport theorem}

Let $(A, m_1)$ and $(B,m_2)$ be finite multisets in an 
(additively written) Abelian
group $G$. We define
$$ 
m_3(c) = \max \big\{ m_1(a)+m_2(b)-1: a\in A, b\in B, a+b=c \big\}
$$ 
the \emph{multiplicity} of an element $c\in A+B$. 
This way $(A+B,m_3)$ becomes a
multiset.

\begin{theorem}
  \label{thm:CD1}
  Let  $(A, m_1)$ and $(B,m_2)$ are multisets from the finite
  prime field $\F_p$. Then we have 
  $$d(A+B)\geq \min \{p, d(A)+d(B)-1 \}. $$
\end{theorem}

\begin{proof}
  We shall use essentially the same polynomial as given in
  \cite{Alon1}. Suppose for contradiction that there exists a multiset
  $C=(C,m)$ such that $A+B\subseteq C$, $p>d(C)$, and $d(C)=d(A)-1+d(B)-1$. 
  We define 
  $$ f(x,y)=\prod_{c\in C}(x+y-c).  $$
  Here we take the factor $(x+y-c)$ precisely $m(c)$ times. We have 
  $f(x,y)\in \F_p[x,y]$ and the coefficient of $x^{d(A)-1}y^{d(B)-1}$ is
  the binomial coefficient $\binom{d(A)-1+d(B)-1}{d(A)-2}$, which is 
  nonzero in $\F_p$. We can apply Theorem~\ref{nonvanish} with $t_1=d(A)-1$, 
  $t_2=d(B)-1$, $(S_1,m_1)=(A,m_1)$ and $(S_2,m_2)=(B,m_2)$. 

  There exist $a\in A$, $b\in B$ and natural numbers $k<m_1(a)$, $l<m_2(b)$ 
  such that in the expansion of $f(x,y)$ at $(a,b)$ the coefficient of 
  $(x-a)^k(y-b)^l$ is nonzero. With the choice $c^*=a+b$ we have 
  $$ f(x,y)=f^*(x,y)(x+y-c^*)^r $$
  where $f^*\in \F_p[x,y]$ and $r\geq m_1(a)+m_2(b)-1$. From 
  $$ (x+y-c^*)^r=\sum_{i=0}^r \binom{r}{i}(x-a)^i(y-b)^{r-i} $$
  we see that $f(x,y)$ vanishes at least $r>k+l$ times at $(a,b)$, a
  contradiction proving the claim.
\end{proof}

\medskip

\begin{remark} The Cauchy Davenport theorem can be proved without 
the polynomial method. Our generalization can also be verified by combining 
the original Cauchy Davenport inequality with  an
elementary argument. In fact, it is possible to prove a bit more.
For a multiset $(Y,m)$ from a group we set 
$$ \deg (Y,m):=\sum_{y\in Y}(m(y)-1).$$ 
We can prove now that
\begin{equation} \label{deg_eq}
  \deg(A+B,m_3) \ge \deg(A,m_1) + \deg(B,m_2).
\end{equation}
If $p\geq |A|+|B|-1$, then we can add to (\ref{deg_eq}) the Cauchy-Davenport 
inquality 
$$ |A+B|\geq |A|+|B|-1 $$
which gives the inequality of Theorem \ref{thm:CD1} under a slightly milder 
condition on $p$.

To prove (\ref{deg_eq}), we may assume 
without loss of generality  that $|A|\le |B|$. Let
  $a_0\in A$ be an element for which $m_1(a_0)$ is maximal. Then
  \begin{gather*}
    \deg(A+B,m_3) \ge \deg\big(a_0+B,m_3\big) =
    \sum_{b\in B} \big(m_3(a_0+b)-1\big) \ge \\
    \ge \sum_{b\in B} \big(m_1(a_0)+m_2(b)-2\big) =
    |B|\cdot \big(m_1(a_0)-1\big) + \sum_{b\in B} \big(m_2(b)-1\big) \ge \\
    \ge |A|\cdot \big(m_1(a_0)-1\big) + \deg(B,m_2) \ge
    \deg(A,m_1) + \deg(B,m_2).
  \end{gather*}

  This multiplicity argument can be extended to non Abelian groups as well. 
From that one 
can  obtain an extension of Theorem~\ref{thm:CD1} by using the generalized
Cauchy Davenport theorem of K\'arolyi \cite{KarolyiCD}. 
\end{remark}


\subsection{Sun's theorem on value sets of polynomials}

In \cite{Sun} Z-W. Sun obtained a common generalization of the Cauchy
Davenport theorem, and the theorem of Felszeghy \cite{Felszeghy} on the
solvability of diagonal equations over finite fields. Here we  give a version 
of Sun's result which involves multiplicities. As before, the original result 
is the special case when every  multiplicity is 1. 

Consider again some nonempty finite multisets  $(S_1, m_1), 
(S_2,m_2), \ldots ,(S_n, m_n)$  
from  a field $\F$,  write $S=S_1\times S_2\times \cdots \times S_n$, 
and let $f(\ve x )\in \Fxv $ be a polynomial. The value
set 
$$ f(S_1, S_2, \ldots , S_n):=\{f(s_1,\ldots ,s_n);~s_1\in S_1,\ldots,
s_n\in S_n \}$$
can be considered as a multiset in $\F$. For a 
$c\in  f(S_1, S_2, \ldots ,S_n)$ we set 
$$ m(c):=\max \{m_1(s_1)+\cdots +m_n(s_n)-n+1;~ \ve s\in S,~ f(\ve s)=c
\}.$$
Let $p(\F)$ denote the characteristic of $\F$ if it is positive, and
set $p(\F)=\infty $ otherwise. 

\begin{theorem} Let $f(\ve x )\in \Fxv $ be a polynomial of the 
form 
$$ f(\ve x ) =a_1x_1^k+a_2x_2^k+\cdots +a_nx_n^k+g(\ve x),$$
where $k$ is a positive integer, $a_1,\ldots ,a_n$ are nonzero 
elements of $\F$, and $g\in \Fxv $ with $\deg g <k$. Also, 
let $(S_1, m_1), (S_2,m_2), \ldots ,(S_n, m_n)$  be nonempty finite
multisets from  $\F$. Then we have 
$$  d(f(S_1, S_2, \ldots , S_n))\geq \min \left\{ p(\F),~\sum_{i=1}^n 
\left\lfloor  \frac{d(S_i)-1}{k}\right\rfloor+1 \right\}. $$
\end{theorem}

\begin{proof}
The argument is an adaptation of the one given by Felszeghy and Sun.
As in \cite{Sun}, after possibly replacing some of the $S_i$ by suitable 
multisubsets $S'_i\subseteq S_i$, we can achieve that $k$ divides 
$d(S_i)-1$ for every $i$, and that  
$\sum\limits_{i=1}^n(d(S_i)-1)=k(N-1)$ holds, where 
$$ N=\min \left\{ p(\F),~\sum_{i=1}^n
\left\lfloor  \frac{d(S_i)-1}{k}\right\rfloor +1\right\}. $$

Now put $C:=  f(S_1, S_2, \ldots , S_n)$, and suppose for contradiction, 
that $d(C)\leq N-1$. Consider the polynomial 
$$ h(\ve x)=f(\stb xn )^{N-1-d(C)}\prod_{c\in C}(f(\stb xn )-c).$$
Here on the right hand side the factor $f(\stb xn )-c$ appears 
exactly $m(c)$ times. The degree 
of $h$ is $N-1$, and the coefficient of the monomial $\ve
y=x_1^{d(S_1)-1}\cdots x_n^{d(S_n)-1}$ in $h(\ve x)$ is the same 
as the coefficient of $\ve y$ in 
$$ (a_1x_1^k+a_2x_2^k+\cdots +a_nx_n^k)^{N-1},$$
which is 
$$  \frac{(N-1)!}{\prod_{i=1}^n((d(S_i)-1)/k)!}a_1^{(d(S_1)-1)/k}\cdots 
a_n^{(d(S_n)-1)/k} \not=0 .$$
By Theorem~\ref{nonvanish} there exists an $\ve s\in S$ such that 
$h(\ve x)\not\in I(\ve s, m(\ve s))$. Let $c^*=f(\stb sn)$. Then $c^*$
appears in the multiset $C$ at least $m= m_1(s_1)+\cdots +m_n(s_n)-n+1$
times, giving that the polynomial 
$$ h^*(\ve x)= (f(\stb xn )-c^*)^m$$
divides  $h(\ve x)$ in $\Fxv$. We expand $h^*(\ve x)$ at $\ve s$. 
As  $f(\stb xn )-c^*$ vanishes at $\ve s$, we obtain 
that $h^*(\ve x)=\sum\limits c_j\ve y_j$, where $c_j\in \F$ and the term $\ve y_j$
is a product of at least $m$ linear factors from the set $\{x_1-s_1,\ldots
x_n-s_n\}$. Thus, for each $j$ there exists an $i$ such that
$(x_i-s_i)^{m_i(s_i)}$ divides $\ve y_j$. We infer that $\ve y_j\in 
I(\ve s, m(\ve s)) $, hence   $h^*(\ve x) \in I(\ve s, m(\ve s))$   and
$h(\ve x) \in I(\ve s, m(\ve s))$ as well. This is a contradiction proving 
the claim $d(C)\geq N$.  
\end{proof}


\subsection{The Eliahou-Kervaire theorem} 

Eliahou and Kervaire \cite{EK} proved an extension of the
Cauchy-Davenport theorem to arbitrary vector spaces over finite prime
fields $\F_p$.

A triple of integers $(r,s,n)$ satisfies the {\em Hopf-Stiefel condition for
the prime $p$} if $\binom{n}{k}$ is divisible by $p$ for every $k$ in the
range $n-r<k<s$. Let  $\beta_p(r,s)$ be the smallest $n$ for which $(r,s,n)$
satisfies the Hopf-Stiefel condition for $p$. We refer to Eliahou and
Kervaire \cite{EK2} for the  properties of the generalized Hopf-Stiefel 
numbers  $\beta_p(r,s)$.  

We have the following extension of the  Eliahou-Kervaire theorem to
multisets.
The proof follows closely the proof of Theorem 5.1 in \cite{Alon1}.  

\begin{theorem}
Let  $(A, m_1)$ and $(B,m_2)$ be multisets from a (finite) vector space $V$ 
over 
the finite prime field $\F_p$, with $d(A)=r$ and $d(B)=s$. Then we have 
$$d(A+B)\geq \beta_p(r,s). $$
\end{theorem}

\begin{proof} We may identify $V$ with a finite field $\F$ of characteristic
$p$, and view $A$ and $B$ as multisets from $\F$. Suppose for contradiction
that $A+B$ is contained in a multiset $C=(C,m)$ such 
that $\beta_p(r,s) >d=d(C)$. As in the proof of Theorem~\ref{thm:CD1}, we 
define 
$$ f(x,y)=\prod_{c\in C}(x+y-c),  $$
where the factor $(x+y-c)$ is taken $m(c)$ times.

From the definition  of $\beta_p(r,s)$ it follows that there exists a $k$ 
with $d-r<k<s$ such that  $\binom{d}{k}$is not divisible by $p$.
This implies, that the coefficient of $x^{d-k}y^k$ in $f$ is nonzero.  
Also, we have $d(A)=r>d-k$ and $d(B)=s>k$. Theorem~\ref{nonvanish}
implies that $f\not \in I(A\times B)$. On the other hand, as in the proof of
Theorem~\ref{thm:CD1}, from the choice of the multiset $C$ we see that 
$f\in I(A\times B)$. This contradiction proves the theorem.
\end{proof}


\end{document}